\documentclass{amsart}

\DeclareMathOperator{\ind}{ind}
\DeclareMathOperator{\id}{id}

\newcommand{\define}[1]{\textbf{#1}}

\newcommand{\Ccal}{\mathcal{C}}
\newcommand{\Fcal}{\mathcal{F}}

\newcommand{\RR}{\mathbb{R}}
\newcommand{\SSS}{\mathbb{S}}
\newcommand{\CP}{\mathbb{CP}}

\newcommand{\Xfrak}{\mathfrak{X}}

\newtheorem{thm}{Theorem}
\newtheorem{lem}[thm]{Lemma}
\newtheorem{prop}[thm]{Proposition}
\theoremstyle{definition}
\newtheorem{defi}[thm]{Definition}

\begin{document}
\title{A generalized Poincar{\'e}-Hopf index theorem}
\author{Beno{\^{\i}}t Jubin}
\date{\today}
\address{Department of Mathematics\\ University of California\\ Berkeley, CA 94720}
\email{jubin@math.berkeley.edu}
\thanks{Research partially supported by NSF grant DMS-0707137.}
\begin{abstract}
We state and prove a generalization of the Poincar{\'e}-Hopf index theorem for manifolds with boundary. We then apply this result to non-vanishing complex vector fields.
\end{abstract}
\maketitle

The Poincar{\'e}-Hopf index theorem, which asserts the equality of the index of a vector field with isolated zeros on a compact manifold and the Euler-Poincar{\'e} characteristic of this manifold, requires, for manifolds with boundary, that the vector field point outwards at the boundary (see for instance~\cite{milnor}). In this article, we state and prove a generalization of this theorem which does not make such a requirement. This theorem is a slight generalization of the main theorem of~\cite{maw} (where the vector field is required to be tangent to the boundary), and the proof uses the same doubling technique. We then use this theorem to prove a result announced in~\cite{jacob} on complex vector fields and the Euler-Poincar{\'e} characteristic.

This generalization of the index theorem to manifolds with boundary deserves to be better known, and has been discovered at least four times in essentially equivalent forms. We include a survey of these prior results in the last section. Even if the present article does not contain new results, we think the more self-contained exposition of the result and the more detailed proofs, especially regarding genericity arguments, are still of interest and more accessible.

In the following, ``manifold'' means ``smooth manifold without boundary'', and ``$\partial$-manifold'' means ``smooth manifold with (possibly) boundary''. Also, we do not attach great importance to the differentiability degree of manifolds and vector fields; e.g. the extensions of vector fields in our proofs are only continuous, but it is easy to smoothen them. Actually, by standard smoothing techniques (see~\cite{hirsch}), the results hold for continuous vector fields on $\Ccal^1$-manifolds. Finally, we use the strong topology on the space of vector fields, which makes it a Baire space (see~\cite{hirsch} for details).

I would like to thank Rob Kirby for his motivating questions and comments and Alan Weinstein for his help in simplifying the proofs of several results.

\section{The generalized Poincar{\'e}-Hopf index theorem}

To state the generalized Poincar{\'e}-Hopf index theorem, we need to define the index of a vector field with isolated zeros on a $\partial$-manifold. To set notations, let $v$ be a vector field with isolated zeros on the $\partial$-manifold $M$ of dimension $n$. We denote by $Z(u)$ the set of zeros of a vector field $u$.

For zeros away from the boundary, this is done in the usual way. That is, the index $\ind(v,m)$ of the vector field $v$ at a zero $m \in M$ is defined to be the degree of the self-map of $\SSS^{n-1}$ obtained by normalizing the vector field image of $v$ by a chart, restricted to a small sphere around the image of $m$. By convention, the degree of the self-map of the empty set is 1. If $v$ has a finite number of zeros, in particular if $M$ is compact, then the \define{index of $v$ on the interior of $M$}, denoted by $\ind_\circ v$, is the sum of the indices of $v$ at its zeros in the interior of $M$.

The index on the boundary is defined as follows. First, we begin with a general definition, and we will keep the same notations in the rest of the article.

\begin{defi}[Collar]
A \define{collar} of a $\partial$-manifold $M$ is a diffeomorphism $\phi=(\phi_1,\phi_2)$ from an open neighborhood $U$ of $\partial M$ to the product $\partial$-manifold $\partial M \times \RR_+$ such that $\phi_1|_{\partial M}=\id_{\partial M}$.
\end{defi}

Note that a collar necessarily satisfies $\phi(\partial M)=\partial M \times \{0\}$. For a proof of the next theorem, we refer to~\cite{hirsch}.

\begin{thm}[Collaring theorem]
Every $\partial$-manifold has a collar.
\end{thm}

Since $\partial M$ is locally compact, the sets
\begin{equation*}
U^f = \{ m \in M \mid \phi_2(m) < f(\phi_1(m)) \}
\end{equation*}
form a neighborhood basis of $\partial M$ when $f$ describes the set $\Fcal = \{ f : \partial M \to \RR_+^* \mid f \mbox{ smooth}\}$ of positive smooth functions on $\partial M$.

If $v$ is a vector field on a $\partial$-manifold $M$ with a collar $\phi$, we can define on the domain $U$ of $\phi$ its \define{tangential} and \define{transverse components} with respect to $\phi$ by
\begin{equation*}
v_\parallel=T\phi_1 \circ v: U \to T(\partial M)
\quad\mbox{and}\qquad
v_\perp=T\phi_2 \circ v: U \to \RR
\end{equation*}
respectively (where for the latter we used the standard identification $T\RR_+ \simeq \RR_+\times~\RR$).

A vector field is called \define{0-transverse} if it is transversal to the zero section of the tangent bundle.

\begin{defi}[Tame vector field]
A vector field on a $\partial$-manifold $M$ with a collaring $\phi$ is \define{$\phi$-tame} (resp. \define{$\phi$-t-tame}) if it has no zero on $\partial M$, and its tangential component with respect to $\phi$ restricted to $\partial M$ has isolated zeros (resp. is 0-transverse).\end{defi}

Since a 0-transverse vector field has isolated zeros, t-tame vector fields are tame. Because $\partial$-manifolds are normal spaces, tame vector fields do not vanish on a neighborhood of the boundary of a $\partial$-manifold. We will also need the following result, a proof of which can be found in~\cite{hirsch} (and for compact manifolds in~\cite{abrah}).

\begin{prop}\label{trans}
The 0-transverse vector fields on a $\partial$-manifold form an open and dense subset of the set of vector fields on this $\partial$-manifold.

\end{prop}

With these definitions, we can state the following proposition.

\begin{prop}
The $\phi$-t-tame vector fields on a $\partial$-manifold $M$ with collar $\phi$ form an open subset of the set of vector fields, dense in each of the sets
\begin{equation*}
X_{w,f}=\{u \in \Xfrak(M) \mid u=w \mbox{ on } M \setminus U^f \mbox{ and } Z(u|_{U^f}) \subseteq \partial M\}
\end{equation*}
for any function $f \in \Fcal$ and vector field $w$ on $M$. In particular, this set is dense in the set of vector fields.

\end{prop}

\begin{proof}
The conditions for t-tameness are open (transversality with a closed submanifold is an open condition, ibid.). We have to prove that they are dense.

First for having no zero on the boundary, that is, we have to ``remove'' the zeros on $\partial M$ of a vector field $v \in X_{w,f}$ without introducing new zeros nor modifying $v$ on $M \setminus U^f$, through an arbitrarily small deformation. This can be done as follows: each zero on the boundary has a neighborhood, included in $U^f$ and arbitrarily small, diffeomorphic through a chart $\psi$ to the half ball $\{x \in \RR^n \mid \|x\|<1 \wedge x_n \geq 0 \}$ and where $v$ does not vanish. We define the continuous vector field $v_1$ on $M$ by $(\psi_*v_1)(x)=(\psi_*v)(x_1,\ldots,x_{n-1},\frac{1+x_n}2)$ on this neighborhood and $v_1=v$ elsewhere. The facts that the neighborhood can be arbitrarily small and that $M$ is locally compact imply that $v_1$ can be chosen arbitrarily close to $v$.

Now for having a tangential component 0-transverse on $\partial M$. By Proposition~\ref{trans}, there is a vector field $w$ on $\partial M$ with isolated zeros and arbitrarily close to $v_\parallel$. We define $v_2$ on $M$ by $v_2=v$ on $M \setminus U^f$ and $\phi_* v_2 = (v_\perp, h v_\parallel + (1-h) w)$ on $U^f$, where $h:\partial M \times \RR_+ \to [0,1]$ is a smooth function such that $h(x,t)=0$ if $t\geq f(x)$ and $h(x,0)=1$. For instance, we can define $h(x,t)=\exp\left(\frac{t}{f(x)(t-f(x))}\right)$ when $t<f(x)$. If $w$ is chosen close enough to $v_\parallel$ then $v_2$ does not vanish. Indeed, because $\partial M$ is locally compact, $v$ has a norm (for any Riemannian metric) locally bounded below by a positive number.

The second statement of the proposition follows from the first since any vector field $u$ with isolated zeros is in a set $X_{u,f}$ for some $f \in \Fcal$ (since $\partial$-manifolds are normal spaces) and vector fields with isolated zeros form a dense subset of the set of vector fields by Proposition~\ref{trans}.
\end{proof}

Now we can define the index of a tame vector field on a $\partial$-manifold with collar at a zero of its tangential component on the boundary.

\begin{defi}[Index on the boundary]
The \textbf{index} of a $\phi$-tame vector field on a $\partial$-manifold with collar $\phi$ at a zero of its tangential component on the boundary is defined to be half the index of its tangential component restricted to the boundary if the vector field points inwards, and minus half it if the vector field points outwards.
\end{defi}

The minus sign in the definition is explained by the following lemma.

\begin{lem}\label{lem1}
If $a$ is a vector field on an open subset $U$ of $\RR^n$ with an isolated zero at $p \in U$, then the index at $(p,0)$ of the vector field $b$ on $U \times (-1,1)$ given by $b(x,t)=(a(x),-t)$, is $\ind(b,(p,0)) = -\ind(a,p)$.
\end{lem}

\begin{proof}
The degree of maps is a morphism (it transforms composition of maps into multiplication of integers) and the reflection with respect to an hyperplane is a self-map of the sphere with degree $-1$.
\end{proof}

If the tangential component has a finite number of zeros on the boundary, in particular if the boundary is compact, then the \define{index on the boundary} of the tame vector field is defined to be the sum of its indices at the zeros of its tangential component on the boundary. If the $\partial$-manifold is compact, this is an integer (by the doubling process explained in the proof of the main theorem below), but need not be an integer if the boundary only is compact. For instance, a non-vanishing vector field on $\RR_+$ has an index on the boundary equal to $\pm 1/2$. As another example, a vector field with isolated zeros on $\CP^2 \times \RR_+$, $\id$-tame and with constant second component, has an index on the boundary equal to $\pm 3/2$, since $\chi(\CP^2)=3$. Notice that this gives another proof that null-cobordant manifolds (through an oriented cobordism) have even Euler-Poincar{\'e} characteristic.

To extend this definition to any vector field with isolated zeros, we need a few more results.

\begin{prop}\label{constInd}
In a $\partial$-manifold with compact boundary, the index on the boundary of a $\phi$-tame vector field is constant in each set $X_{w,f}$.
\end{prop}

\begin{proof}
Consider two $\phi$-t-tame vector fields $v_0$ and $v_1$ in $X_{w,f}$. They are homotopic within $X_{w,f}$ since $\partial U^f$ (where the symbol $\partial$ is used in its general topology meaning) is a deformation retract of $U^f$. Indeed, we define a homotopy by $v_t = \lambda^{2t}_*v_0$ if $t\in [0,\frac12]$ and $v_t = \lambda^{2(1-t)}_*v_1$ if $t\in [\frac12,1]$, where the deformation $\lambda^t: U^f \to U^f$ is defined by
\begin{equation*}
\phi \circ \lambda^t = (\phi_1,(1-t)\phi_2 + t f)
\end{equation*}
for any $t \in [0,1]$.

Now consider the map  $H: \partial M \times [0,1] \to T(\partial M)$ defined by $H(m,t)={v_t}_\parallel(m)$. Up to a small perturbation with support disjoint from the boundary of $\partial M \times [0,1]$, we can suppose it is 0-transverse, so that $H^{-1}(0)$ is a compact one-dimensional $\partial$-manifold, neat as a submanifold of $\partial M \times [0,1]$ (by 0-transversality of ${v_i}_\parallel$, $i\in \{0,1\}$).

Consider a connected component $W$ of $H^{-1}(0)$ with non-empty boundary $\partial W=\{a_1,a_2\}$, with $a_1\neq a_2$, $a_k \in \partial M \times \{i_k\}$, $i_k \in \{0,1\}$. We consider a small hollow tube $N$ around $W$, whose boundary consists of two $(n-2)$-dimensional spheres, around $a_1$ and $a_2$. By a standard result (\cite[Lemma~1.2 p.~123]{hirsch}), and since $N$ is compact, the degree of $H/\|H\|$ (for any Riemannian metric) on $\partial N$ is zero. Therefore, if $i_1 = i_2$, then $\ind ({v_{i_1}}_\parallel,a_1)+\ind ({v_{i_2}}_\parallel,a_2)=0$, and if $i_1 \neq i_2$, then $\ind ({v_{i_1}}_\parallel,a_1)=\ind ({v_{i_2}}_\parallel,a_2)$ (because of the orientation reversal). But on $W$, ${v_t}_\parallel$ vanishes, so that ${v_t}_\perp$ has to keep a constant sign, hence $\ind (v_{i_1},a_1)+\ind (v_{i_2},a_2)=0$ or $\ind (v_{i_1},a_1)=\ind (v_{i_2},a_2)$ respectively. Therefore, grouping the boundary points of $H^{-1}(0)$ into pairs, we get $\ind_\partial v_0=\ind_\partial v_1$.

Finally, a $\phi$-tame vector field can be approximated by a $\phi$-t-tame one by a standard construction consisting in exploding a degenerate zero into non-degenerate ones, and the degree being a topological invariant, the indices on the boundary of both vector fields are the same.
\end{proof}

We now have to prove that the index on the boundary is independent of the collaring. This follows from the next two propositions.

\begin{prop}
For any two collars and any set $X_{w,f}$, the set of vector fields which are t-tame for both collars is dense in $X_{w,f}$.
\end{prop}

\begin{proof}
This is simply because the intersection of two open dense sets is dense.
\end{proof}

\begin{prop}
A vector field on a $\partial$-manifold with compact boundary which is tame for two collars has the same index on the boundary for both collars.
\end{prop}

\begin{proof}
Let $\phi$ and $\psi$ be two collars of $M$ and $v$ be a vector field which is t-tame for both (we can assume this for the same reason as in the proof of Proposition~\ref{constInd}). To apply the same argument as in the proof of Proposition~\ref{constInd}, we need to show that $v^\phi = T\phi \circ v$ and $v^\psi = T\psi \circ v$ are homotopic on $\partial M \times \{0\}$ through non-vanishing vector fields. We will actually show that the convex combination $v^t=tv^\phi + (1-t) v^\psi$ does not vanish. Since collars agree on the boundary, $v^\psi = (v^\phi_1 + v^\phi_2 w, \alpha v^\phi_2)$ on $\partial M$, where $w \in \Xfrak(\partial M)$ and $\alpha \in \Fcal$, so that
\begin{equation*}
v^t=(v^\phi_1 + (1-t) v^\phi_2 w, \left( t +  (1-t) \alpha \right) v^\phi_2).
\end{equation*}
If at a point of the boundary $v^t_2=0$, then $v^\phi_2=0$, so $v^t_1=v^\phi_1 \neq 0$.
\end{proof}

These results enable us to define the index on the boundary of a vector field $v$ with isolated zeros on a $\partial$-manifold $M$ with compact boundary. Indeed, $v$ is in a set $X_{v,f}$ for some $f \in \Fcal$, and we define its index on the boundary $\ind_\partial v$ to be the index on the boundary of any $\phi$-tame vector field in this set for any collar $\phi$ of $M$.

Finally, the index of a vector field with isolated zeros on a $\partial$-manifold is defined to be the sum of its indices on the interior and on the boundary:
\begin{equation*}
\ind v = \ind_\circ v + \ind_\partial v.
\end{equation*}

Using Lemma~\ref{lem1}, the following proposition is clear.

\begin{prop}
Let $v$ be a vector field with isolated zeros on a compact oriented $\partial$-manifold $M$. Then $\ind_\circ (-v) = (-1)^{\dim M} \ind_\circ v$ and $\ind_\partial (-v) = (-1)^{\dim M} \ind_\partial v$, so that
\begin{equation*}
\ind (-v) = (-1)^{\dim M} \ind v.
\end{equation*}
\end{prop}

With the above definition, we can state the generalized Poincar{\'e}-Hopf index theorem.

\begin{thm}
Let $v$ be a vector field with isolated zeros on the compact oriented $\partial$-manifold $M$. Then
\begin{equation*}
\ind v =
\begin{cases}
\chi(M)	& \mbox{if $\dim M$ is even,}\\
0			& \mbox{if $\dim M$ is odd.}
\end{cases}
\end{equation*}
\end{thm}

\begin{proof}

We fix a collar $\phi$ for $M$ and suppose without loss of generality that $v$ is $\phi$-tame. We extend the $\partial$-manifold $M$ to $\tilde{M}=M \cup \partial M \times [-1,0]$, with identification along $\id \times \{-1\}:\partial M \to \partial M \times \{-1\}$, and the vector field $v$ to a continuous vector field $\tilde{v}$ on $\tilde{M}$ by $\tilde{v}(m,t)=v_\parallel(m) - t v_\perp(m)$, for any $m \in \partial M$ and $t \in [-1,0]$. Then we consider the double $\widetilde{M}$ of $\tilde{M}$ that we obtain by gluing $\tilde{M}$ to a copy of itself along the identity of its boundary. The vector field $\tilde{v}$ extends to $\widetilde{v}$, equal to $\tilde{v}$ on both copies of $\tilde{M}$. This can be done since $\tilde{v}$ is tangential on the boundary of $\tilde{M}$. The continuous vector field $\widetilde{v}$ has isolated zeros on $\widetilde{M}$ and its index is twice the index of $v$ on $M$. Indeed, $v$ being $\phi$-tame, the zeros of $\widetilde{v}$ are in the embedded images of $\partial \tilde{M}$ and of the two copies of $M$. The sum of the indices at the latter zeros is obviously $2\ind_\circ v$. For the former zeros, Lemma~\ref{lem1} applied to $\partial \tilde{M} \times (-1,1)$ and the definition of the index of $v$ on the boundary show that the sum of the indices at those zeros is $2\ind_\partial v$.

We now apply the Poincar{\'e}-Hopf index theorem to the vector field $\widetilde{v}$ on the compact manifold $\widetilde{M}$. The manifold $\widetilde{M}$ has the same dimension as $M$ and has no boundary, so if $\dim M$ is odd, $\chi(\widetilde{M})=0$, and if $\dim M$ is even, the manifold $\partial M$ has odd dimension and no boundary, so $\chi(\widetilde{M}) = 2 \chi(\tilde{M}) - \chi(\partial M) = 2 \chi(M)$. Since $\ind \widetilde{v} = 2\ind v$, the theorem is proved.
\end{proof}

Notice that we could have defined the index of a vector field with isolated zeros by means of a perturbation (so that it is tame on the boundary for a given collar) and extensions to $\tilde{M}$ and then to its double. The theorem would have proved that the index on the boundary was independent of the choices of the perturbation and extensions. Although the exposition above was longer, we think it is worth carrying out because it yields a more intrinsic definition of the index on the boundary (and does not require the $\partial$-manifold to be compact).

This theorem is a generalization of the Poincar{\'e}-Hopf index theorem for manifolds without boundary, since the Euler-Poincar{\'e} characteristic of an odd dimensional manifold without boundary is zero. Let us check that it also yields the Poincar{\'e}-Hopf index theorem for $\partial$-manifolds. That is, we have to check that if $v$ points outwards on $\partial M$, then $\ind_\partial v=0$ if $n$ is even and $\ind_\partial v=-\chi(M)$ if $n$ is odd. But our definition gives in this case $\ind_\partial v = -\frac12 \chi(\partial M)$, which is equal to 0 if $n$ is even and $-\chi(M)$ if $n$ is odd, as wanted.

As an example, let $v$ be a non-vanishing vector field on a manifold $M$. Consider the $\partial$-manifold $N=M \times [-1,1]$ and the vector field $w$ defined by $w(m,t)=(\cos \frac{\pi t}2 v(m),\sin \frac{\pi t}2)$ on $N$. It is non-vanishing on $N$ and points outwards on its boundary, so $\ind_\partial v = -\frac12 \chi(\partial M) = - \chi(M) = 0$.

As another example, let us consider two vector fields on the closed unit ball of $\RR^n$, a constant non-zero vector field and the radial vector field given by $v(x)=\|x\|x$. For the former, the index in the interior is 0. The tangential component on the boundary is a vector field going from one point to its antipode along great circles. At the point of divergence, the vector field points inwards, giving an index $1/2$, and at the point of convergence, the vector field points outwards, giving an index $(-1)^n/2$. Therefore, the index on the boundary is 0 when $n$ is odd and 1 when $n$ is even. For the latter, the index on the interior is the index at 0, which is 1 (the degree of the identity map of $\SSS^{n-1}$), and the index on the boundary is, from above, $-1$ when $n$ is odd and 0 when $n$ is even. So we see that while $\ind$ is an invariant of the manifold, neither $\ind_\circ$ nor $\ind_\partial$ need be.

\section{Complex vector fields and the Euler-Poincar{\'e} characteristic}

A \define{complex vector field} on a manifold is a section of the complexified tangent bundle. If the manifold has a Riemannian metric, then the \define{square norm} of a complex tangent vector is defined by 
\begin{equation*}
|\xi+i\eta|^2=|\xi|^2-|\eta|^2+2i\langle\xi,\eta\rangle.
\end{equation*}

The article~\cite{jacob} states the following theorem and gives an incomplete proof. Indeed, on page~2, line~18, the fact that the same alternative holds at all points of the region requires \textit{a priori} that the region be connected, which is not necessarily the case (see however the last paragraph of this section).

\begin{thm}
A compact oriented manifold bearing a complex vector field with non-vanishing square norm has Euler-Poincar{\'e} characteristic zero.
\end{thm}

\begin{proof}
Let $v=\xi + i \eta$ be a complex vector field with non-vanishing square norm on the compact manifold $M$. If $\dim M$ is odd, there is nothing to prove, so suppose $\dim M$ is even. We can also suppose that $\xi$ has isolated zeros, since this is a generic property. We want to show that $\ind \xi =0$ (actually, a perturbation of $\xi$), since this will imply by the usual Poincar{\'e}-Hopf index theorem that $\chi(M)=0$.

Dealing separately with each connected component of $M$, we can assume that $M$ is connected. We partition $M$ into the three subsets $A_+ = \{m \in M \mid |\xi_m|>|\eta_m|\}$, $B = \{m \in M \mid |\xi_m|=|\eta_m|\}$ and $A_- = \{m \in M \mid |\xi_m|<|\eta_m|\}$. If $B$ is empty, then either $A_-$ or $A_+$ is empty, by the intermediate value theorem. Now $\xi$ does not vanish on $A_+ \cup B$ and $\eta$ does not vanish on $A_- \cup B$, so if $A_-$ or $A_+$ is empty, $M$ carries a non-vanishing vector field, so $\chi(M)=0$. So let us suppose that none of these three sets is empty.

There exists an open subset $\Omega$ of $M$ with smooth boundary such that $B \subset \Omega \subset \overline{\Omega} \subset \{m \in M \mid \langle \xi_m, \eta_m \rangle \neq 0\}$. Then $\overline{\Omega} \cap A_-$ is a $\partial$-manifold. We denote by $\Sigma_-$ its boundary. By the collaring theorem, there is a neighborhood $U$ of $\Sigma_-$ in $\overline{\Omega} \cap A_-$ diffeomorphic to $\Sigma_- \times [0,2)$. Let us perturb $\xi$ on this set in the following way: on each connected component of this set either $\langle \xi, \eta \rangle >0$ or $\langle \xi, - \eta \rangle >0$. So set $\xi_1(m,t)=(1-t(2-t))\xi \pm t(2-t)\eta$ for any $m \in \Sigma_-$ and $t \in [0,2)$, depending on the case occurring, and $\xi_1=\xi$ out of this set. Thus $\xi_1$ has the same zeros as $\xi$. Now, apply the generalized Poincar{\'e}-Hopf index theorem on $N = (A_- \setminus \Omega) \cup \Sigma_+ \times [0,1)$ to $\xi_1$ and $\eta$. We get
\begin{equation*}
\ind_\circ \xi_1 + \ind_\partial \xi_1 = \chi(N) = \ind_\circ \eta + \ind_\partial \eta.
\end{equation*}
Since $\eta$ does not vanish on $N$, $\ind_\circ \eta = 0$. On each component of $\partial N$, we have $\xi_1 = \pm \eta$, so $\ind_\partial \xi_1 = \ind_\partial \eta$ or $\ind_\partial \xi_1 = (-1)^{\dim M} \ind_\partial \eta = \ind_\partial \eta$, so that $\ind_\circ \xi_1=0$. Since $\xi_1$ does not vanish on $A_+ \cup \Omega$, $\ind \xi_1=0$ on $M$, so by the usual Poincar{\'e}-Hopf index theorem, $\chi(M)=0$.
\end{proof}

Actually, the proof in~\cite{jacob} can be fixed in the following way, as H. Jacobowitz noted. The argument of the original proof does not give a non-vanishing vector field but only one up to sign, so that it still defines a line field. But the existence of a line field on a compact oriented manifold is enough to imply that the manifold has Euler-Poincar\'e characteristic zero. Indeed, let $M$ be a compact oriented manifold and let $M'$ be the intersection (in $TM$) of the line field and the unit tangent bundle (for any Riemannian metric). Then $M'$ is a double cover of $M$, so it is compact and oriented, and it carries a non-vanishing vector field, namely the restriction of the tautological section of $T(TM)$. Therefore $M'$ has Euler-Poincar\'e characteristic zero and so does $M$.

\section{Similar generalizations of the Poincar{\'e}-Hopf index theorem}

The first extension of the Poincar{\'e}-Hopf index theorem for general vector fields on $\partial$-manifolds dates back to Morse (1929) who proved in~\cite{morse} (the result had been announced in~\cite{morse1}) that for any tame vector field $v$ on a compact oriented $\partial$-manifold $M$,
\begin{equation*}
\ind_\circ(v) + \ind(\partial_- v) = \chi(M)
\end{equation*}
where $\partial_- v$ is the tangential component of the vector field $v$ restricted to the open submanifold of $\partial M$ where $v$ points (strictly) inwards. Following Gottlieb, we call this identity the Morse index formula. This formula was rediscovered by Pugh (1968) in~\cite{pugh}, and then by Gottlieb (1986) in~\cite{gott1} and~\cite{gott2}, and generalized by Gottlieb and Samarayanake in~\cite{sama} for non-everywhere defined vector fields.

The Morse index formula is easily seen to be equivalent to the formula given in the present paper. Indeed, with obvious notations, one has
\begin{equation*}
\ind_\partial(v) = \frac12 \big(\ind(\partial_- v) - \ind(\partial_+ v)\big)
\end{equation*}
so that proving their equivalence amounts to proving that $\ind(\partial_+ v) + \ind(\partial_- v) = 2\chi(M)$ if $\dim M$ is odd and $0$ if $\dim M$ is even. But by the usual Poincar\'e-Hopf index theorem, and because $v$ has no zero on $\partial M$, the left-hand side is equal to $\chi(\partial M)$, which is precisely $2\chi(M)$ if $\dim M$ is odd and $0$ if $\dim M$ is even, as wanted.

\end{document}